\documentclass[10pt]{article}
\usepackage[margin=1in]{geometry}
\usepackage{graphicx}
\usepackage{color}
\usepackage{amsmath}
\usepackage{amsthm}

\usepackage{array}
\usepackage{ascmac}
\usepackage{amssymb}
\usepackage{booktabs}
\usepackage{wrapfig}
\usepackage[algoruled,linesnumbered,algo2e,vlined]{algorithm2e}
\usepackage{algorithmicx}
\usepackage{algpseudocode}
\usepackage{url}
\usepackage{authblk}
\usepackage[nocompress]{cite}
\usepackage{multirow}
\usepackage{multicol}
\usepackage{cases}
\usepackage[normalem]{ulem}
\usepackage{boites}

\newcommand{\calF}{{\mathcal{F}}}
\newcommand{\calS}{{\mathcal{S}}}

\usepackage[pagewise,mathlines]{lineno}

\usepackage[draft,mode=multiuser]{fixme}
\usepackage{colortbl}

\fxsetup{theme=color}
\definecolor{fxtarget}{rgb}{0.0000,0.0000,0.4823}

\fxuseenvlayout{colorsig}
\fxusetargetlayout{color}

\fxuseenvlayout{colorsig}
\fxusetargetlayout{color}
\FXRegisterAuthor{hs}{ahs}{HS}
\FXRegisterAuthor{tm}{atm}{TM}
\FXRegisterAuthor{dk}{adk}{DK}
\FXRegisterAuthor{si}{asi}{SI}
\fxsetface{env}{}

\numberwithin{equation}{section}
\theoremstyle{plain}
\newtheorem{theorem}{Theorem}[section]
\newtheorem{lemma}[theorem]{Lemma}
\newtheorem{fact}[theorem]{Fact}

\theoremstyle{definition}

\newenvironment{keywords}
  {\par\smallskip\noindent\textbf{Keywords.}\ }
  {\par\smallskip}
\providecommand{\ackname}{Acknowledgements}
\begin{document}

\title{Relaxation of Square-Freeness}

\author[1]{Hiroki~Shibata}
\author[2]{Takuya~Mieno}
\author[3]{Dominik~K\"oppl}
\author[4]{Shunsuke~Inenaga}

\affil[1]{Joint Graduate School of Mathematics for Innovation, Kyushu University\\
\protect\url{shibata.hiroki.753@s.kyushu-u.ac.jp}}
\affil[2]{Graduate School of Informatics and Engineering, University of Electro-Communications\\
\protect\url{tmieno@uec.ac.jp}}
\affil[3]{Department of Computer Science and Engineering, University of Yamanashi\\
\protect\url{dkppl@yamanashi.ac.jp}}
\affil[4]{Department of Informatics, Kyushu University\\
\protect\url{inenaga.shunsuke.380@m.kyushu-u.ac.jp}}

\date{}

\maketitle

\begin{abstract}
We extend the analysis of nonrepetitive sequences of Entringer et al. [Journal of Combinatorial Theory, 1974] to relaxations of equality testing under nonstandard equivalence relations, in particular parameterized equivalence and order-preserving equivalence.
For this setting, we introduce $\ell^+$-squares, defined as squares whose total length is at least $2\ell$.
We obtain an infinite $3^+$-parameterized-square-free ternary word and an infinite $3^+$-order-preserving-square-free binary word through an approach based on combinatorics on words.
In addition, we report the longest $\ell^+$-square-free words across several equivalence relations.
\end{abstract}


\section{Introduction}
A \emph{square} is a non-empty word of the form $uu$.
Squares are one of the simplest and most studied repetitions in words~\cite{crochemore95squares,fraenkel98square}.
A word is \emph{square-free} if it contains no square as a subword.
More than a century ago, Thue~\cite{thue1906} constructed an infinite square-free word over a ternary alphabet.
A ternary alphabet is the smallest possible alphabet for an infinite square-free word.
Indeed, every infinite binary sequence contains at least one of the squares $aa$, $bb$, and $abab$.

The notion of a square can be generalized by replacing equality between the two halves with a weaker equivalence relation on words.
In this paper, we use two equivalence relations from pattern matching:
\emph{parameterized equivalence}~\cite{baker96parametrized} and \emph{order-preserving equivalence}~\cite{kim14orderpreserving}.
We say that two words $u$ and $v$ of the same length~$\ell$ are \emph{parameterized equivalent} if there is a bijection $f$ with $f(u[i])=v[i]$ for each text position $i$.
We say that they are \emph{order-preserving equivalent} if $u[i]\preceq u[j]\Leftrightarrow v[i]\preceq v[j]$ holds for every pair of text positions $i$ and $j$.
Order-preserving equivalence can also be viewed through the \emph{increasing bijection} on the alphabet.
Equivalently,
$u$ and $v$ are parameterized equivalent under an increasing bijection.
A non-empty word $uv$ is called a \emph{parameterized} (resp.\ \emph{order-preserving}) \emph{square} if $u$ and $v$ are parameterized (resp.\ order-preserving) equivalent~\cite{kociumaka16maximum}.

Some equivalence relations admit no infinite square-free words.
For example, every word of length two is a parameterized square, so no parameterized-square-free infinite word can exist.
Even if we consider only \emph{non-trivial} squares (those of length greater than two), no infinite word avoids every non-trivial parameterized square~\cite{kociumaka16maximum}.
Consequently, weaker notions of square-freeness are required to give positive results on infinite words that avoid a weaker type of square.
We study \emph{$\ell^+$-squares}, squares of length at least $2\ell$, and investigate \emph{$\ell^+$-square-free words} under parameterized equivalence.

A complete characterization of the avoidability of length-restricted parameterized and order-preserving squares can be derived by combining existing results~\cite{kociumaka16maximum,NgORS19}.
However, the proofs by Ng et al.~\cite{NgORS19} do not rely solely on combinatorics on words because their approach uses large morphisms and automated theorem proving tools.
In this paper, we construct square-free words using an approach based on combinatorics on words.

Our main results are summarized as follows.
We first prove the existence of an infinite $3^+$-parameterized-square-free ternary word.
We then show the existence of an infinite $3^+$-order-preserving-square-free binary word.
Both infinite words are constructed using small morphisms, and their square-freeness is proved through an approach based on combinatorics on words.
Finally, through computational experiments, we establish the maximum lengths of words that avoid length-restricted squares under several nonstandard equivalence relations and give explicit longest words.

Table~\ref{tab:longest_squarefree} summarizes the maximum lengths of $\ell^+$-square-free words under several equivalence relations.

\begin{table}[t]
    \centering
    \caption{
    Lengths of the longest $\ell^+$-square-free words under various equivalence relations and alphabet sizes.
    An entry marked $\infty$ indicates the existence of an infinite $\ell^+$-square-free word.
    A question mark ($?$) denotes a case that remains open.
    A dash ($-$) indicates that infiniteness follows from another case (e.g., a smaller $\ell$ or $\sigma$).
    }
    \vspace{1em}

    \begin{minipage}[t]{0.48\textwidth}
        \centering
        \textbf{Strict Equality}\\[0.5em]
        \begin{tabular}{c|cccc}
         $\sigma \backslash \ell$ & 1 & 2 & 3 & 4 \\
         \hline
         2 & 3 & 18 & $\infty$ & $-$ \\
         3 & $\infty$ & $-$ & $-$ & $-$ \\
         4 & $-$ & $-$ & $-$ & $-$ \\
         5 & $-$ & $-$ & $-$ & $-$ \\
        \end{tabular}
    \end{minipage}
    \hfill
    \begin{minipage}[t]{0.48\textwidth}
        \centering
        \textbf{Parameterized Equivalence}\\[0.5em]
        \begin{tabular}{c|cccc}
         $\sigma \backslash \ell$ & 1 & 2 & 3 & 4 \\
         \hline
         2 & 1 & 7 & $\infty$ & $-$ \\
         3 & 1 & 9 & $-$ & $-$ \\
         4 & 1 & 9 & $-$ & $-$ \\
         5 & 1 & 9 & $-$ & $-$ \\
        \end{tabular}
    \end{minipage}

    \vspace{1em}

    \begin{minipage}[t]{0.48\textwidth}
        \centering
        \textbf{Order-Preserving Equivalence}\\[0.5em]
        \begin{tabular}{c|cccc}
         $\sigma \backslash \ell$ & 1 & 2 & 3 & 4 \\
         \hline
         2 & 1 & 7 & $\infty$ & $-$  \\
         3 & 1 & $\infty$ & $-$ & $-$ \\
         4 & 1 & $-$ & $-$ & $-$ \\
         5 & 1 & $-$ & $-$ & $-$ \\
        \end{tabular}
    \end{minipage}
    \hfill
    \begin{minipage}[t]{0.48\textwidth}
        \centering
        \textbf{Cartesian-Tree Equivalence~\cite{park19cartesian}}\\[0.5em]
        \begin{tabular}{c|cccc}
         $\sigma \backslash \ell$ & 1 & 2 & 3 & 4 \\
         \hline
         2 & 1 & 5 & 29 & $?$ \\
         3 & 1 & 9 & $?$ & $?$ \\
         4 & 1 & 9 & $?$ & $?$ \\
         5 & 1 & 9 & $?$ & $?$ \\
        \end{tabular}
    \end{minipage}
    \label{tab:longest_squarefree}
\end{table}

\paragraph{Related Work.}
Square-free words were first studied by Thue~\cite{thue1906} over a century ago and have since remained an active topic of research.
Variants involving length-restricted squares have also been explored.
For example, it is known that every binary word of length greater than eighteen contains a $2^+$-square~\cite{entringer74nonrepetitive}.
In contrast, infinite binary $3^+$-square-free words exist~\cite{entringer74nonrepetitive,rampersad05avoiding}; they are constructed by applying a morphism to a ternary square-free word.
This kind of freeness has also been studied for Gray codes~\cite{prodinger83nonrepetitive} and circular words~\cite{currie21characterization}.

Squares and other repetitions under generalized equivalence relations on words have also been studied.
In the abelian model, where equivalence is defined by the multiset of characters, Keränen~\cite{keranen92abelian} showed the existence of infinite abelian-square-free words over a four-character alphabet, but it is known that none exist over a ternary alphabet.
In the parameterized setting, the number of parameterized squares in a word has also been studied~\cite{kociumaka16maximum,hamai24parameterized}.
Cubic and unary patterns under permutations, which are closely related to parameterized matching, have also been studied~\cite{manea15cubic,currie18unary}.

Words avoiding parameterized and order-preserving squares have also been studied in several works.
Kociumaka et al.~\cite{kociumaka16maximum} showed that an infinite $2^+$-order-preserving-square-free word exists over a ternary alphabet, whereas no such word exists in the parameterized model even without restrictions on the alphabet size.
Ng et al.~\cite{NgORS19} constructed an infinite binary word that avoids every ordinary square and \emph{antisquare} of length at least $6$, where an antisquare is a binary word of the form $x\overline{x}$ and $\overline{x}$ is obtained by changing every character of $x$.
Since parameterized squares over a binary alphabet are precisely ordinary squares and antisquares, this infinite binary word is $3^+$-parameterized-square-free.
Moreover, every order-preserving square is a parameterized square, and hence the word is $3^+$-order-preserving-square-free.
Their results were proved using large morphisms found through computer search and \texttt{Walnut}~\cite{Walnut}.
Combining these results gives a complete characterization of the values of $\ell$ and the alphabet sizes for which infinite $\ell^+$-square-free words exist in the parameterized and order-preserving models.

\section{Preliminaries}
Let $\Sigma$ be an \emph{alphabet}.
We denote the size of $\Sigma$ by $\sigma$.
An element of $\Sigma$ is referred to as a \emph{character}, and an element of $\Sigma^*$ is referred to as a \emph{word}.
We denote a word $x$ of length $n$ as $x = x[0] \cdots x[n-1]$, where $x[i]$ denotes the $i$-th character of $x$.
For $0 \leq i < j \leq |x|$, we denote by $x[i..j)$ the subword of $x$ consisting of the characters $x[i] \cdots x[j-1]$,
where the notation $[i..j)$ follows the half-open interval convention.
If a word $x$ does not contain another word $w$ as a subword, we say that $x$ \emph{avoids} $w$.

Two words $x, y \in \Sigma^n$ are \emph{parameterized equivalent} if there exists a bijective function $f: \Sigma \rightarrow \Sigma$ such that $f(x[i]) = y[i]$ for all $0 \leq i < n$.
Two words $x, y \in \Sigma^n$ on a totally ordered alphabet $\Sigma$ are \emph{order-preserving equivalent} if $x[i]\preceq x[j]\Leftrightarrow y[i]\preceq y[j]$ holds for every pair $i,j \in [0..n-1]$.
We note that order-preserving equivalence is stronger than parameterized equivalence;
in particular, if two words $x$ and $y$ are order-preserving equivalent, then they are also parameterized equivalent.

A word of the form $ww$ is called a (standard) \emph{square}.
A word that contains no standard square as a subword is called \emph{square-free}.
For an integer $\ell$, we call such a square an $\ell^+$\emph{-square} if its total length is at least $2\ell$.
A word of the form $uv$, where $u$ and $v$ are \emph{parameterized} (resp. \emph{order-preserving}) \emph{equivalent}, is called a \emph{parameterized} (resp. \emph{order-preserving}) \emph{square}.
We use the term \emph{$\ell^+$-parameterized} (resp. \emph{$\ell^+$-order-preserving}) \emph{square} to refer to such a square when its length is at least $2\ell$.
A word is said to be \emph{$\ell^+$-square-free}, with respect to an equivalence model (exact, parameterized, order-preserving, etc.), if it contains no square of the corresponding type with length at least $2\ell$.

A \emph{morphism} on two alphabets $\Sigma$ and $\Sigma'$ is a function {$h: \Sigma^* \rightarrow \Sigma'^*$} such that $h(xy) = h(x)h(y)$ for all words $x, y \in \Sigma^*$.
Hence, to define a morphism $h$, it suffices to specify $h(c)$ for each character $c \in \Sigma$.
When $\Sigma' = \Sigma$ (i.e., $h: \Sigma^* \to \Sigma^*$), we define $h^k(x)$ for $x \in \Sigma^*$ as the word obtained by applying $h$ to $x$ repeatedly $k$ times.
If $x$ is a prefix of $h(x)$, then $h^k(x)$ is a prefix of $h^{k+1}(x)$ for every $k \geq 0$.
If the lengths $|h^k(x)|$ are unbounded, these finite words define an infinite word.
Let $\Sigma_2 = \{0, 1\}$ and $\Sigma_3 = \{\rm a, b, c\}$ denote the binary and ternary alphabets, respectively.
In this paper, we put special emphasis on the morphism $\phi: \Sigma_3^* \rightarrow \Sigma_3^*$ defined by

\begin{align*}
    {\rm a} &\mapsto {\rm abc}, \\
    {\rm b} &\mapsto {\rm ac}, \\
    {\rm c} &\mapsto {\rm b}.
\end{align*}
It is known that $\phi^k({\rm a})$ is square-free for any $k \geq 0$~\cite{cori90abelian}.
For instance,
$\phi({\rm a}) = {\rm abc}$,
$\phi^2({\rm a}) = {\rm abcacb}$, and
$\phi^3({\rm a}) = {\rm abcacbabcbac}$,
none of which contains a square.

\section{$3^+$-Parameterized-Square-Free Ternary Words}
In this section, we aim to construct an infinite $3^+$-parameterized-square-free ternary word.
We define $C = {\rm cbbbc}$ and let $\mu: \Sigma_3^* \rightarrow \Sigma_3^*$ be the morphism that maps $\rm c$ to $C$ while leaving $\rm a$ and $\rm b$ unchanged.
For each $k \ge 0$, we define $x'_k = \phi^k({\rm a})$ and $x_k = \mu(x'_k)$.
Recall that $x'_k$ is square-free~\cite{cori90abelian}.

We first show that $x_k$ contains no $3^+$-parameterized squares except for the standard squares.
\begin{lemma} \label{lem:no_parameterized}
For any $k \ge 0$, if a $3^+$-parameterized square $w = uv$ $(|u|=|v| \geq 3)$ exists as a subword of $x_k$, then $u = v$.
\end{lemma}
\begin{proof}
Let $w = uv$ be a $3^+$-parameterized square subword of $x_k$ of length $2\ell$ ($\ell \geq 3$), where $u$ and $v$ are parameterized equivalent under some bijection $f$.

We first determine the nonempty subwords of $x'_k$ that avoid $\rm c$.
Since every subword of the square-free word $x'_k$ is square-free, such a subword cannot contain any of the binary squares $\rm aa$, $\rm bb$, $\rm abab$, and $\rm baba$.
Indeed, avoiding $\rm aa$ and $\rm bb$ forces a binary word to alternate, and every alternating word of length at least four contains $\rm abab$ or $\rm baba$.
Thus, any nonempty subword of $x'_k$ that avoids $\rm c$ belongs to $\{{\rm a}, {\rm b}, {\rm ab}, {\rm ba}, {\rm aba}, {\rm bab}\}$.
We next exclude $\rm aba$.
This is clear for $x'_0={\rm a}$.
For $k \geq 1$, write $x'_k=\phi(z)$ with $z=x'_{k-1}$.
In any occurrence of $\rm aba$, the middle $\rm b$ is preceded and followed by $\rm a$.
Consider an occurrence of $\rm b$ in $\phi(z)$.
It is either the middle character of $\phi({\rm a})={\rm abc}$ or the whole image $\phi({\rm c})={\rm b}$, since $\phi({\rm b})={\rm ac}$ contains no $\rm b$.
In the first case, this $\rm b$ is followed by $\rm c$, and hence it cannot be the middle character of $\rm aba$.
In the second case, if this $\rm c$ is the first character of $z$, then this $\rm b$ has no preceding character and cannot be the middle character of $\rm aba$.
Otherwise, the character preceding this $\rm b$ is the last character of $\phi(d)$ for the character $d$ preceding this $\rm c$ in $z$.
Since $\phi({\rm a})={\rm abc}$, $\phi({\rm b})={\rm ac}$, and $\phi({\rm c})={\rm b}$ end with $\rm c$, $\rm c$, and $\rm b$, respectively, the character preceding this $\rm b$ is not $\rm a$.
Hence $\rm aba$ does not occur in $x'_k$.
Consequently, every nonempty subword of $x'_k$ that avoids $\rm c$ belongs to
$S = \{{\rm a}, {\rm b}, {\rm ab}, {\rm ba}, {\rm bab}\}$.

We next show that $\rm bb$ occurs in $w$.
Since $x_k$ is obtained from $x'_k$ by replacing each occurrence of ${\rm c}$ with $C$,
any subword of $x_k$ that avoids $\rm bb$ cannot contain an occurrence of $C$, because $C = \rm cbbbc$ contains $\rm bb$.
Moreover, no suffix or prefix of $C$ of length at least $3$ avoids $\rm bb$.
Suppose, for the sake of contradiction, that $w$ avoids $\rm bb$.
Then, it must be of the form $w = \alpha s \beta$, where $\alpha$ (resp. $\beta$) is a suffix (resp. prefix) of $C$ of length at most $2$ and $s \in S$.
This gives $|s| = 2\ell - |\alpha| - |\beta| \ge 2$, so
the only possibilities for $s$ are $\rm ab$, $\rm ba$, and $\rm bab$.
Consequently, the only possible candidates for $w$, whose length is even and at least six,
are $\rm bcbabc$, $\rm cbabcb$, $\rm bcabcb$, and $\rm bcbacb$.
None of them is a parameterized square.
Hence, $\rm bb$ occurs in $w$.

Since $|w| \ge 6$ and every occurrence of $\rm bb$ in $x_k$ must be a part of an image of $\rm c$ (i.e., $C=\rm cbbbc$),
either $u$ or $v$ must contain the subword $\rm cbb$ or $\rm bbc$.
Since $x'_k$ is square-free, it contains no two equal consecutive characters.
The morphism $\mu$ only replaces $\rm c$ with $C=\rm cbbbc$.
Thus, every occurrence of two equal consecutive characters in $x_k$ is contained in the block $\rm bbb$ of an occurrence of $C$.
Hence, if a length-three subword of $x_k$ has its last two (resp. first two) characters equal and the remaining character different, then it must be $\rm cbb$ (resp. $\rm bbc$).
Thus, the only subword of $x_k$ parameterized equivalent to $\rm cbb$ (resp. $\rm bbc$) is $\rm cbb$ itself (resp. $\rm bbc$ itself).
Therefore, the corresponding subword in the other half of $w$ is identical to $\rm cbb$ or $\rm bbc$, and hence $f({\rm b}) = {\rm b}$ and $f({\rm c}) = {\rm c}$.
Consequently, $f$ must be the identity function, and $w$ is a standard square $w = uu$.
\end{proof}

\begin{figure}[t]
    \centering
    \includegraphics[width=0.8\linewidth]{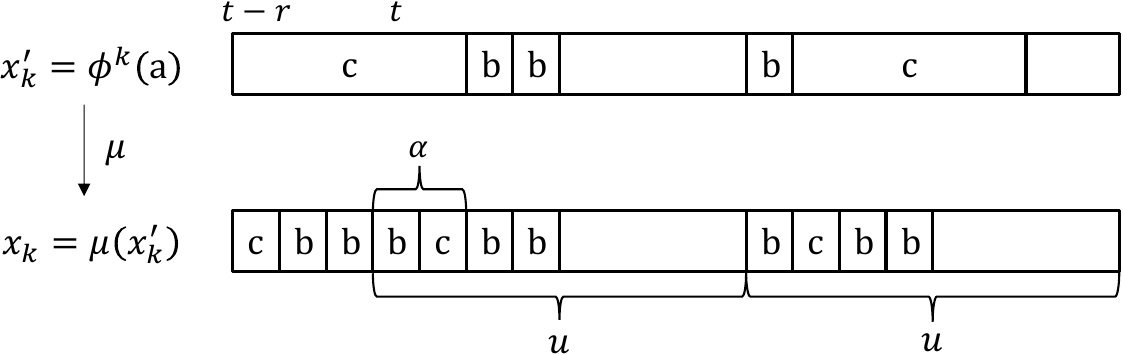}
    \caption{
    Illustration for Lemma~\ref{lem:standard_square_free}.
    In this figure, the two occurrences of $u$ in $w = u^2$ have different offsets.
    The left one has offset $r = 3$, and the right one starts at a boundary, that is, with offset $0$.
    In this case, $u$ starts with $\alpha {\rm bb} = {\rm bcbb}$, where $\alpha = {\rm bc}$.
    The factorization starting at position $t - r$ is $C|{\rm b}|{\rm b}$.
    This implies that $x'_k$ contains the square $\rm bb$.
    }
    \label{fig:offset_diff}
\end{figure}

We also show the $3^+$-square-freeness of $x_k$.
\begin{lemma} \label{lem:standard_square_free}
    For any $k \geq 0$,
    the word $x_k$ is $3^+$-square-free.
\end{lemma}
\begin{proof}
Let $\calS = \{{\rm a}, {\rm b}, C\}$.
In this proof, we define a \emph{factorization} of a word $z$ as a decomposition $z = f_1 \cdots f_m$ with $f_i \in \calS$ for every $i$.
The words $f_i$ are called \emph{factors}.
Since $\calS$ is prefix-free, such a factorization is unique whenever it exists.
Moreover, since $\mu$ maps $\rm c$ to $C$ while leaving $\rm a$ and $\rm b$ unchanged,
the word $x_k = \mu(x'_k)$ has a factorization, and hence it has a unique factorization.
We call a position a \emph{boundary} if a factor starts there.
If a position is not a boundary, then it lies inside an occurrence of $C$.
In this case, we denote its offset from the beginning of that occurrence to be an element of $\{ 1, 2, 3, 4\}$.
Boundary positions are said to have offset $0$.
If both endpoints of a subword $z$ of $x_k$ are boundaries, then its factorization determines a unique word $v$ such that $\mu(v) = z$.

For the sake of contradiction, assume that there exists a square subword $w = uu = x_k[p-\ell..p+\ell)$ of $x_k$ with $\ell \geq 3$ centered at text position $p$.
Let $i = p - \ell$ be the starting position of the square.

We first show that the two occurrences $x_k[i..p)$ and $x_k[p..p+\ell)$ of $u$ start at the same offset.
Assume that one occurrence starts at offset $r \in \{1, 2\}$.
Then it begins with ${\rm bbb}$ if $r = 1$ and with ${\rm bbc}$ if $r = 2$.
These prefixes occur only inside an occurrence of $C$, and the offset is uniquely determined.
Hence the other occurrence also starts at the same offset $r$.

Next assume that one occurrence starts at offset $r \in \{3,4\}$.
Let $\alpha = C[r..5)$, so $\alpha = {\rm bc}$ for $r=3$ and $\alpha = {\rm c}$ for $r=4$.
If an occurrence starts away from a boundary, then the prefix $\alpha$ determines its offset uniquely.
Suppose, for the sake of contradiction, that the other occurrence starts at a different offset.
Since the prefix $\alpha$ determines every non-boundary offset, the other occurrence must start at a boundary.
In this case, the other occurrence of $u$ starts at offset $0$ with prefix $\alpha$.
If $r = 4$ $(\alpha = {\rm c})$, then the first factor in the factorization beginning at this occurrence must be $C$.
If $r = 3$ $(\alpha = {\rm bc})$, then the first and second factors in this factorization must be ${\rm b}$ and $C$, respectively.
In both cases, $u$ starts with $\alpha {\rm bb}$.
Consider the occurrence of $u$ with offset $r$, and let its starting position be $t$.
Since $\alpha$ appears as a suffix of the occurrence of $C$ starting at position $t - r$,
the factorization of $x_k$ at position $t - r$ begins with $C|{\rm b}|{\rm b}$.
This factorization implies that $x'_k$ contains the square ${\rm bb}$, contradicting the square-freeness of $x'_k$~(see Figure~\ref{fig:offset_diff}).
Thus, whenever one occurrence of $u$ starts away from a boundary, the other occurrence starts at the same offset in $\{1, 2, 3, 4\}$.
Equivalently, if one occurrence starts at a boundary, then the other occurrence also starts at a boundary.

If both occurrences of $u$ start at boundaries, then both endpoints of the first occurrence $x_k[i..p)$ are boundaries.
The equality of the two occurrences and the uniqueness of the factorization imply that $x_k[p..p+\ell)$ has the same factorization.
Thus, both endpoints of $x_k[p..p+\ell)$ are also boundaries.
Let $v$ be the unique word such that $\mu(v) = x_k[i..p)$.
Then both endpoints of $w$ are boundaries, and $w = \mu(v^2)$.
Hence $v^2$ occurs in $x'_k$.
This contradicts the square-freeness of $x'_k$.
Otherwise, let $r \in \{1, 2, 3, 4\}$ be the common offset of the two occurrences of $u$.
Let $\alpha$ be the suffix of $C$ of length $5 - r$ and $\beta$ be the prefix of $C$ of length $r$.
Since $\beta \alpha = C$ and the center position $p$ of the square $w = uu$ lies on an occurrence of $C$ with offset $r = |\beta|$,
$u$ can be written as $u = \alpha z \beta$, where $z = x_k[i+|\alpha|..p-|\beta|)$~(see Figure~\ref{fig:square_contradiction}).
Since the positions $i+|\alpha|$, $p-|\beta|$, $p+|\alpha|$, and $p+\ell-|\beta|$ are boundaries,
the two subwords $x_k[i+|\alpha|..p-|\beta|)$ and $x_k[p+|\alpha|..p+\ell-|\beta|)$ have exactly the same factorization.
Let $v$ be the unique word such that $\mu(v) = z$.
Hence, $w = \alpha \mu(v) \beta \alpha \mu(v) \beta$.
Since both the starting and ending position of $w$ lie inside occurrences of $C$,
by extending this subword to the left by $r$ positions and to the right by $5 - r$ positions, we obtain
$C \mu(v) C \mu(v) C$ as a subword of $x_k$.
This subword is equal to $\mu({\rm c} v {\rm c} v {\rm c})$.
Thus, $x'_k$ contains the square $({\rm c} v)^2$.
This contradicts the square-freeness of $x'_k$.
\end{proof}

\begin{figure}[t]
    \centering
    \includegraphics[width=1.0\linewidth]{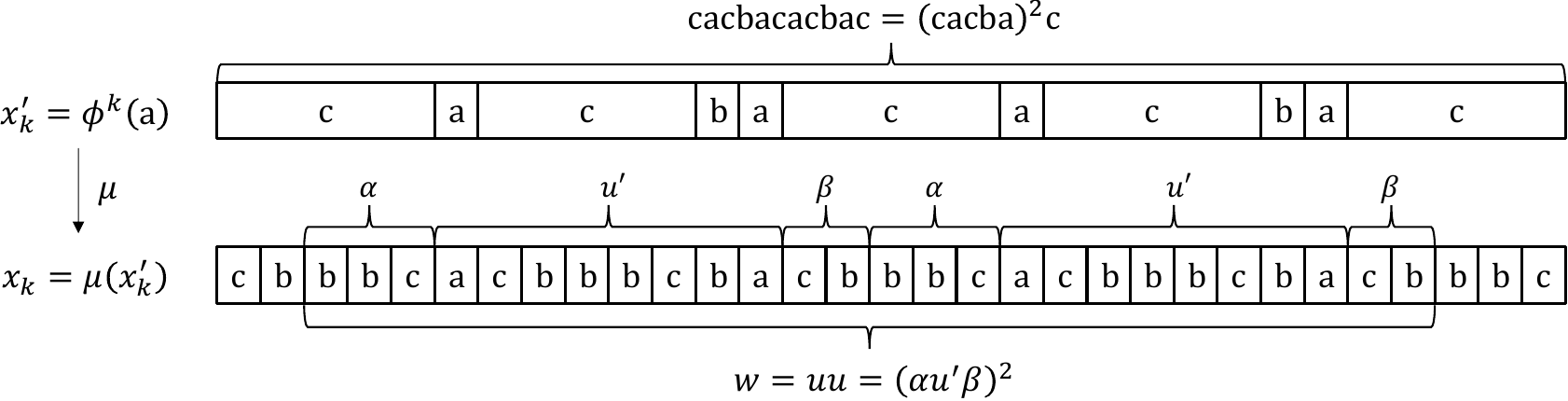}
    \caption{
    Illustration for Lemma~\ref{lem:standard_square_free}.
    In the figure, $x_k$ has a square $w = uu = ({\rm bbcacbbbcbacb})^2$ and $u$ is written as $u=\alpha z \beta$
    where $\alpha = {\rm bbc}, \beta = {\rm cb}$, and $z = {\rm acbbbcba}$.
    In this case, $w = uu$ is generated from $({\rm cacba})^2 c$ and thus $x'_k$ has a square $({\rm cacba})^2$.
    }
    \label{fig:square_contradiction}
\end{figure}

By combining the above results, we can obtain the main result.
\begin{theorem} \label{thm:parameterized_square_freeness}
    For any $k \geq 0$,
    the ternary word $x_k = \mu(\phi^k({\rm a}))$ is $3^+$-parameterized-square-free.
\end{theorem}
\begin{proof}
By Lemma~\ref{lem:no_parameterized},
if a subword $w$ is a $3^+$-parameterized square, then $w$ must be a standard square.
However, by Lemma~\ref{lem:standard_square_free}, such square does not exist.
Therefore $x_k$ does not contain any $3^+$-parameterized-square.
\end{proof}

\section{$3^+$-Order-Preserving-Square-Free Binary Words} \label{app:3plus}
In this section, we show that there exists an infinite $3^+$-order-preserving-square-free binary word.
The proof is similar to \cite{prodinger79infinite}, where it is observed that applying the morphism $\psi:\Sigma_3^* \rightarrow \Sigma_2^*$ with $\psi({\rm a}) = 0000, \psi({\rm b}) = 0101$, and $\psi({\rm c}) = 1111$ to a square-free ternary word~$x$ yields a $3^+$-square-free binary word (in the strict sense).
We show that the same property holds when we slightly modify the morphism $\psi$ to another morphism $\tau:\Sigma_3^* \rightarrow \Sigma_2^*$~(Lemma~\ref{lem:square_free}).
Then, we construct an infinite binary word by combining the two morphisms $\tau$ and $\phi$~(Theorem~\ref{thm:3_order-preserving}).

Let $\tau:\Sigma_3^* \rightarrow \Sigma_2^*$ be the morphism defined as
\begin{align*}
{\rm a} &\mapsto 00, \\
{\rm b} &\mapsto 0101, \\
{\rm c} &\mapsto 11.
\end{align*}
The following lemma will be proved in Section~\ref{sec:lemmaproof}.
\begin{lemma} \label{lem:square_free}
For any square-free ternary word $x$, the word $\tau(x)$ is $3^+$-square-free.
\end{lemma}

Let $x'_k = \phi^{k}({\rm a})$ and $y_k = \tau(x'_k)$ for every $k \ge 0$.
Recall that each $x'_k$ is square-free, and thus, each $y_k$ is $3^+$-square-free by Lemma~\ref{lem:square_free}.
We prove the main theorem of this section.
\begin{theorem} \label{thm:3_order-preserving}
    For any $k \geq 0$,
    the binary word $y_k = \tau(\phi^k({\rm a}))$ is $3^+$-order-preserving-square-free.
\end{theorem}
\begin{proof}
Assume, for the sake of contradiction, that there exists a subword $w = uv$ $(|u|=|v|)$ of $y_k$ that is a $3^+$-order-preserving square of length $2\ell$ ($\ell \geq 3$).
Since $y_k$ is $3^+$-square-free and $\ell \geq 3$, we have $u \neq v$.
If either $u$ or $v$ contains both $0$ and $1$, then the only order-preserving bijection between their alphabets is the identity.
Therefore, since $u \neq v$, the only possible case is when both $u$ and $v$ consist of a single repeated character but differ from each other (i.e., $\{u, v\} = \{0^\ell, 1^\ell\}$).
To prove that such a square cannot occur in $y_k$, it suffices to show that $y_k$ contains neither $0^3 1^3$ nor $1^3 0^3$ as a subword.

Since $\tau({\rm a})=00$, $\tau({\rm b})=0101$, and $\tau({\rm c})=11$, any occurrence of $0^3$ (resp. $1^3$) in $y_k=\tau(x'_k)$ is contained in $\tau({\rm aa})$ or $\tau({\rm ab})$ (resp. $\tau({\rm cc})$ or $\tau({\rm bc})$), if it exists.
Since $x'_k$ is square-free, it contains neither $\rm aa$ nor $\rm cc$.
Thus, $0^3$ can appear only as the prefix of $\tau({\rm ab})=000101$, and $1^3$ can appear only as the suffix of $\tau({\rm bc})=010111$.
Therefore, a subword of the form $0^3 1^3$ or $1^3 0^3$ can occur only inside $\tau({\rm bcab})$, so that
\[
\tau({\rm bcab}) = \tau({\rm b})\tau({\rm c})\tau({\rm a})\tau({\rm b}) = 010\,1^3 0^3\,101.
\]
However, $\rm bcab$ cannot occur in $x'_k$.
If $k=0$, then this is immediate from $x'_0={\rm a}$.
If $k \ge 1$, then both $\rm bc$ and $\rm ab$ can only be generated by applying $\phi$ to the character $\rm a$.
Thus, an occurrence of $\rm bcab$ in $x'_k=\phi(x'_{k-1})$ would imply an occurrence of $\rm aa$ in $x'_{k-1}$, contradicting the square-freeness of $x'_{k-1}$.
Therefore, $y_k$ contains no $3^+$-order-preserving square.
\end{proof}

In the rest of this section, we give a proof of Lemma~\ref{lem:square_free}.

\subsection{Proof of Lemma \ref{lem:square_free}}\label{sec:lemmaproof}

Before starting the proof, we introduce some notation.
Let $x$ be a square-free ternary word.
We denote $w = \tau(x)$.
A subword $w[i..j)$ of $w$ is called a \emph{block} if it is exactly generated from the subword $x[s..e)$,
i.e., there exists a pair of integers $s$ and $e$ such that $\tau(x[0..s)) = w[0..i)$ and $\tau(x[s..e)) = w[i..j)$ hold.
We denote $\calF = \{ \tau({\rm a}), \tau({\rm b}), \tau({\rm c}) \} = \{ 00, 0101, 11\}$.
In this subsection, we define a \emph{factorization} of a block $z$ as a decomposition $z = f_1 \cdots f_m$ with $f_i \in \calF$ for every $i$.
The words $f_i$ are called \emph{factors}.
Note that this convention differs from the factorization used in Lemma~\ref{lem:standard_square_free}.
Since $\calF$ is prefix-free, such a factorization is unique whenever it exists.
Every block $w[i..j)$ has such a factorization by definition, and therefore its factorization is uniquely defined.
Moreover, the factorization of a block $z$ determines a unique word $v$ such that $\tau(v) = z$.
A position $i$ with $0 \leq i \leq |w|$ is called a \emph{boundary} if $i = |w|$ or a factor starts at position $i$ in the factorization of $w$.
For any position $i$, let $o_i$ and $o'_i$ be the smallest non-negative integers such that $i + o_i$ and $i - o'_i$ are boundaries, respectively, where $0 \leq o_i, o'_i \leq 3$.
In particular, $o_i = 0$ if and only if $o'_i = 0$.
Since the length of a factor is either two (for $00$, $11$) or four (for $0101$),
$o_i \geq 2$ implies that the factor ending at position $i + o_i$ is $0101$.
Similarly, $o'_i \geq 2$ implies that the factor starting at position $i - o'_i$ is $0101$.

Next, we introduce Fact~\ref{fac:short_square_free}, which states that $w$ contains no short squares of lengths between $6$ and $12$.
Let $w[i.. j)$ be a subword of $w$ of length $2\ell$.
By construction of $w$, the subword $w[i-o'_i.. j+o_j)$ is a block generated from a square-free subword of $x$ of length at most $\ell + 1$.
Indeed, each factor has length at least two, so a subword of length $2\ell$ can intersect at most $\ell+1$ factors.
Thus, we can verify Fact~\ref{fac:short_square_free} by enumerating all words in $\{ \tau(v) \mid v \in \Sigma_3^*, 1 \leq |v| \leq 7, \text{$v$ is square-free} \}$ and checking $3^+$-square-freeness\footnote{The source code is available at \url{https://github.com/koeppl/squarechecker/check_short_string_3plus_square_freeness.py}}.

\begin{fact} \label{fac:short_square_free}
    For every $3 \le \ell \le 6$, the word $w$ contains no square of length $2\ell$.
\end{fact}

We next establish two lemmas, which rule out long equal subwords with different boundary offsets.
\begin{lemma} \label{lem:mod2_neq_substring}
There is no pair of subwords $w[i..i + \ell)$ and $w[j..j + \ell)$ such that $\ell \geq 7$, $w[i..i + \ell) = w[j..j + \ell)$, and $o_{i} \not\equiv o_{j} \pmod 2$.
\end{lemma}
\begin{proof}
Assume on the contrary that there are two subwords $w[i..i + \ell) = w[j..j + \ell)$ of length $\ell \ge 7$ with $o_{i} \not\equiv o_{j} \pmod 2$.
Let $u = w[i..i+\ell) = w[j..j+\ell)$. Assume $o_i < o_j$ without loss of generality.
We first consider the case $o_i = 0, o_j = 3$.
In this case, $w[j-1..j+3)$ must be a factor, specifically $w[j-1..j+3) = 0101$.
Then, $w[j..j+3) = w[i..i+3) = 101$, which contradicts $o_i=0$ since no factor starts with $101$.

The remaining case is $o_j = o_i + 1$.
Let $f_i$ and $f_j$ be the factors starting at positions $i + o_i$ and $j + o_j$, respectively.
Since $o_j = o_i + 1$, we have $0 \leq o_i \leq 2$.
Together with $\ell \geq 7$ and the fact that every factor has length at most four, this ensures that the factor immediately following $f_i$ exists.
Let $g_i$ be this factor.
Since $\ell \ge 7$, both $f_i$ and $f_j$ are contained in $u = w[i..i+\ell) = w[j..j+\ell)$.
Further, since $o_j = o_i + 1$, the starting positions of the occurrences of $f_i$ and $f_j$ inside $u$ are different by one position.
These facts imply that $(f_ig_i)[1..|f_j|+1) = f_j$ holds.

We consider the following three cases according to the value of $f_i$.

\noindent \textbf{Case 1: $\boldsymbol{f_i = 0101}$:}
Then $f_j$ must start with $10$ because
$f_j[0..2) = u[o_j..o_j+2) = u[o_i+1..o_i+3) = f_i[1..3) = 10$.
However, no factor starts with $10$, a contradiction.

\noindent \textbf{Case 2: $\boldsymbol{f_i = 11}$:}
Then $f_j$ must start with $1$, hence $f_j = 11 \in \calF$.
From $(f_ig_i)[1..|f_j|+1) = f_j = 11$, $g_i$ also starts with $1$, so $g_i = 11$.
Therefore, $f_i g_i = 1111$, which is the image of $\mathrm{cc}$ under $\tau$.
This implies that $x$ contains the square $\mathrm{cc}$, contradicting the square-freeness of $x$.

\noindent \textbf{Case 3: $\boldsymbol{f_i = 00}$:}
Then $f_j$ must start with $0$, so $f_j$ is either $00$ or $0101$.
If $f_j = 00$, as in the previous case,
we obtain $g_i = 00$
from $(f_ig_i)[1..|f_j|+1) = f_j = 00$.
Then, $f_i g_i = 0000 = \tau(\mathrm{aa})$,
contradicting the square-freeness of $x$.
Otherwise, if $f_j = 0101$, then $g_i$ must start with $10$, which contradicts that $g_i \in \calF$.
\end{proof}

\begin{lemma} \label{lem:mod4_diff2_substring_pref}
There is no pair of subwords $w[i..i + \ell)$ and $w[j..j + \ell)$ such that $\ell \geq 7$, $w[i..i + \ell) = w[j..j + \ell)$, and $o_i - o_j \equiv 2 \pmod 4$.
\end{lemma}
\begin{proof}
Assume on the contrary that there are two subwords $w[i..i + \ell) = w[j..j + \ell)$ of length $\ell \ge 7$ with $o_{i} - o_j \equiv 2 \pmod 4$.
Let $u = w[i..i+\ell) = w[j..j+\ell)$. Assume $o_i < o_j$ without loss of generality.
Then,
the only possible pairs for $(o_i, o_j)$ are $(0, 2)$ and $(1, 3)$, so $o_j = o_i + 2$.
Let $f_i$ and $f_j$ be the factors starting at positions $i + o_i$ and $j + o_j$, respectively.
Also, let $e_j$ be the factor immediately preceding $f_j$.
Since $o_j \ge 2$, we have $e_j = 0101$.

From $u = w[i..i+\ell) = w[j..j+\ell)$ and $o_j = o_i + 2$, it follows that
$u[o_i..o_i+2) = u[o_j-2..o_j) = e_j[|e_j|-2.. |e_j|) = 01$.
Hence $f_i = 0101$, because it is the only factor that starts with $01$.
Also,
$u[o_j..o_j+2) = u[o_i+2..o_i+4) = f_i[2..4) = 01$,
so $f_j = 0101$.
Therefore, the two consecutive factors around position $j+o_j$ satisfy
$e_j = f_j = 0101$.
Consequently, $e_j f_j = 01010101 = \tau({\rm bb})$, which implies that $x$ contains the square ${\rm bb}$,
contradicting the square-freeness of $x$.
\end{proof}

Now, we are ready to prove
Lemma~\ref{lem:square_free}.

\begin{proof}[of Lemma~\ref{lem:square_free}]
For the sake of contradiction, assume $w$ contains a square $w[i..k)$ of length $2\ell \ge 6$. Let $j = i + \ell$ where $\ell = (k-i)/2$, so $w[i..k) = uu$ with $u = w[i..j) = w[j..k)$.
By Fact~\ref{fac:short_square_free}, $w$ has no square with $3 \leq \ell \leq 6$, thus $\ell \geq 7$.
Let $d = (o_i - o_j) \bmod 4$.
We classify squares in $w$ by the values of $o_i, o'_j, o_j, o'_k$ and $d$.

\noindent \textbf{Case 1: $\boldsymbol{d = 2}$:}
This case contradicts Lemma~\ref{lem:mod4_diff2_substring_pref}.

\noindent \textbf{Case 2: $\boldsymbol{d  \in \{ 1, 3\}}$:}
This case is equivalent to $o_i \not\equiv o_j \pmod{2}$. Thus, the existence of square $w[i..k) = uu$ contradicts Lemma~\ref{lem:mod2_neq_substring}.

\noindent \textbf{Case 3: $\boldsymbol{d = 0}$:}
In this case, $o_i = o_j$ holds.
Assume $w[i..j)$ is factorized as $w[i..j) = \alpha f_1 \cdots f_t \beta$, where $\alpha = w[i..i+o_i)$ and $\beta = w[j-o'_j..j)$ may be empty, and each $f_p$ is a factor.
Then, $w[j..k)$ is also factorized as $w[j..k) = \alpha f_1 \cdots f_t \beta$, because $w[i + o_i..j) = w[j + o_j..k)$, both $i + o_i$ and $j + o_j$ are boundaries, and the factorization of a block is unique.
More precisely, $o'_j = o'_k$ and $w[i..k) = \alpha f_1 \cdots f_t \beta \alpha f_1 \cdots f_t \beta$ hold.
We denote $f = f_1 \cdots f_t$, and let $g$ be the word such that $\tau(g) = f$.
Note that such $g$ is uniquely defined.

\textbf{Case 3-(a): $\boldsymbol{o_i = o_j = 0}$:}
In this case, both $i$ and $j$ are boundaries.
Hence, both $\alpha$ and $\beta$ are the empty word, and $w[i..j)$ is fully factorized as $w[i..j) = f_1 \cdots f_t = f$.
Likewise, $w[j..k) = f$, and thus $w[i..k) = w[i..j)w[j..k) = f^2$.
Since $\tau(g^2) = f^2$, the block $w[i..k)$ is generated from the word $g^2$.
Hence $g^2$ occurs in $x$, a contradiction.

\textbf{Case 3-(b): $\boldsymbol{o_i = o_j \in \{ 2, 3\}}$:}
In this case, the factors crossing positions $i$ and $j$ must be $0101$.
In particular, $w[i-o'_i..i+o_i) = w[j-o'_j..j+o_j) = 0101$.
Thus, the block $w[i-o'_i..k-o'_k)$ is factorized as $|0101|f|0101|f|$.
Therefore, this block is generated from the word $({\rm b} g)^2$, implying that $x$ contains a square
(see the top of Figure~\ref{fig:case2ab}), a contradiction.

\textbf{Case 3-(c)-(i): $\boldsymbol{o_i = o_j = 1 \text{ and } o'_j = o'_k = 1}$:}
In this case, the factor crossing position $j$ has length two, i.e.,  $w[j-o'_j..j+o_j)$ is either $00$ or $11$.
Let $c$ be the character such that $w[j-o'_j..j+o_j) = cc$.
Since $w[j-o'_j]$ and $w[j]$ correspond to the last and first characters of $u$, respectively, we have $\alpha = \beta = c$.
Therefore, $w[i..k) = uu$ is of the form $c f c c f c$.
If $c = 0$, the factor crossing position $i$ must be $00$, because it is the only factor that ends with $0$.
Then, the block $w[i-o'_i..k-o'_k)$ is factorized as $|00|f|00|f|$.
Therefore, this block is generated from the word $({\rm a} g)^2$, implying that $x$ contains a square (see the top of Figure~\ref{fig:case2d}), a contradiction.
If $c = 1$, the factor crossing position $k$ must be $11$, because it is the only factor that starts with $1$.
Then, the block $w[i+o_i..k+o_k)$ is factorized as $|f|11|f|11|$.
Therefore, this block is generated from the word $(g {\rm c})^2$, implying that $x$ contains a square (see the bottom of Figure~\ref{fig:case2d}), a contradiction.

\textbf{Case 3-(c)-(ii): $\boldsymbol{o_i = o_j = 1 \text{ and } o'_j = o'_k \neq 1}$:}
In this case, $o'_j = o'_k = 3$.
This is because $o_j = 1$, so $j$ is not a boundary and the factor crossing position $j$ has length $o_j + o'_j = 1 + o'_j$.
Every factor has length either $2$ or $4$.
The proof of this case is analogous to that of Case 3-(b).
In this case, the factors crossing positions $j$ and $k$ are $0101$, so the block
$w[i+o_i..k+o_k)$ is factorized as $|f|0101|f|0101|$.
Therefore, this block is generated from the word $(g {\rm b})^2$, implying that $x$ contains a square
(see the bottom of Figure~\ref{fig:case2ab}), a contradiction.

\begin{figure}[t]
    \centering
    \includegraphics[width=0.9\linewidth]{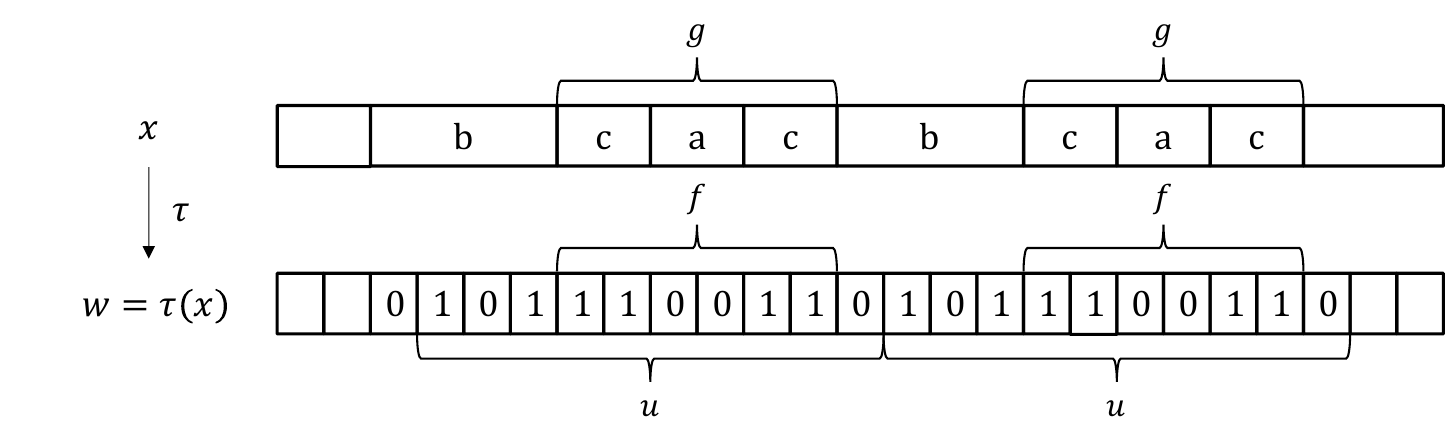}
    \includegraphics[width=0.9\linewidth]{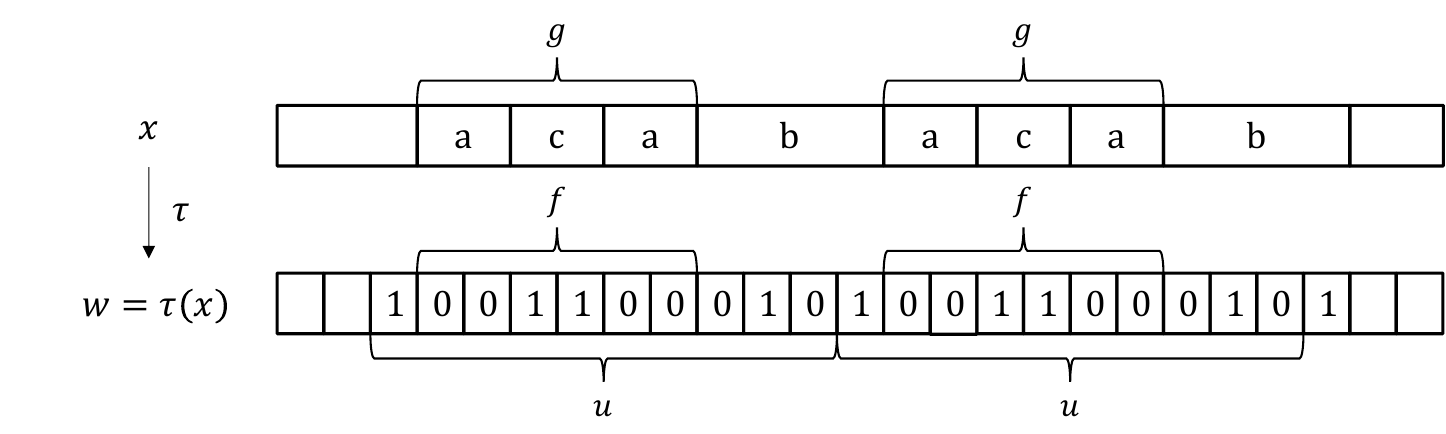}
    \caption{%
    Illustrations of examples for Case 3-(b) (top) and Case 3-(c)-(ii) (bottom) of Lemma~\ref{lem:square_free}.
    In the first figure, the word $w = \tau(x)$ contains a square $uu = (1011100110)^2$ and $x$ contains a square $({\rm b}g)^2 = ({\rm bcac})^2$.
    In the second figure, the word $w = \tau(x)$ contains a square $uu = (1001100010)^2$ and $x$ contains a square $(g{\rm b})^2 = ({\rm acab})^2$.
    }
    \label{fig:case2ab}
\end{figure}

\begin{figure}[t]
    \centering
    \includegraphics[width=0.9\linewidth]{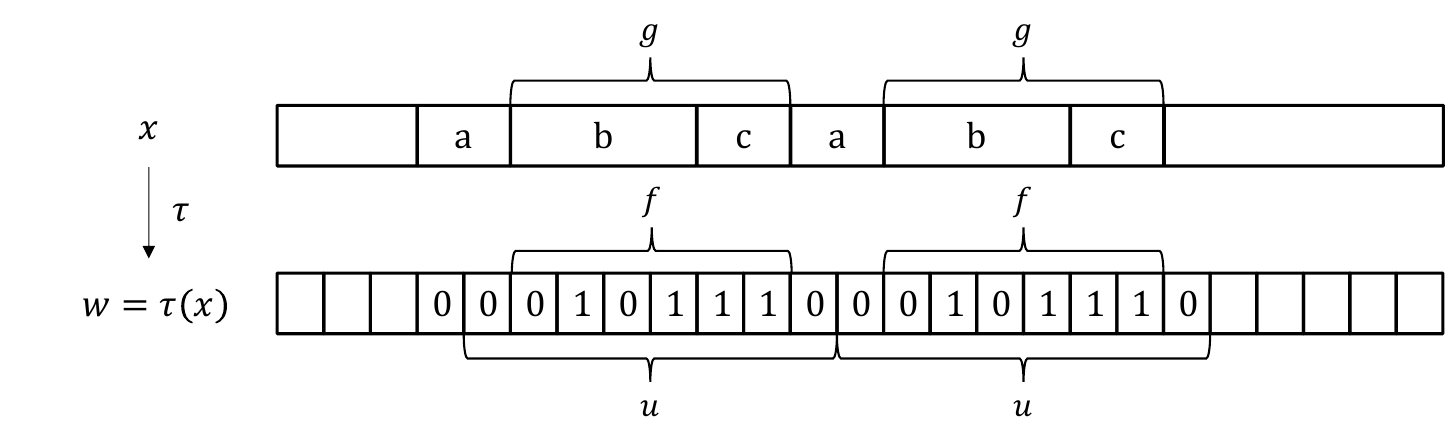}
    \includegraphics[width=0.9\linewidth]{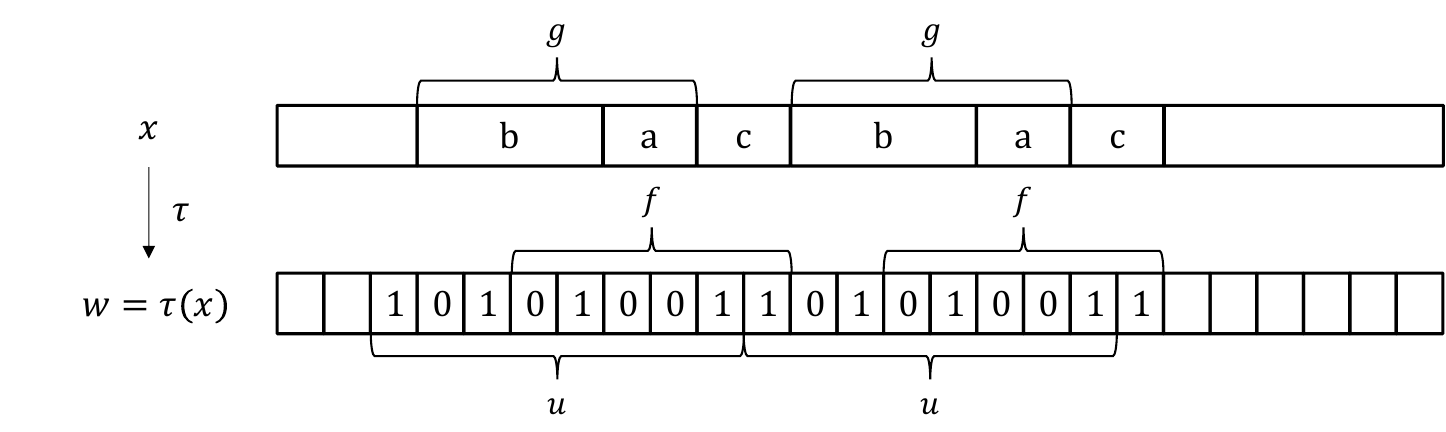}
    \caption{%
    Illustrations of examples for Case 3-(c)-(i) of Lemma~\ref{lem:square_free}.
    In the first figure, the word $w = \tau(x)$ contains a square $uu = (00101110)^2$ and $x$ contains a square $({\rm a}g)^2 = ({\rm abc})^2$.
    In the second figure, the word $w = \tau(x)$ contains a square $uu = (10101001)^2$ and $x$ contains a square $(g{\rm c})^2 = ({\rm bac})^2$.
    }
    \label{fig:case2d}
\end{figure}

The preceding cases exhaust all possible scenarios for a square $uu$ occurring in $w = \tau(x)$.
Since each of these cases leads to a contradiction with the square-freeness of $x$, we conclude that $w = \tau(x)$ is $3^+$-square-free.
\end{proof}

\section{Conclusions}
In this paper, we studied $\ell^+$-square-free words under several equivalence relations.
We gave combinatorial proofs of the existence of an infinite $3^+$-parameterized-square-free ternary word and an infinite $3^+$-order-preserving-square-free binary word.
Both infinite words were constructed using small morphisms.

Through computational experiments, we also determined the maximum lengths in the finite cases reported in Table~\ref{tab:longest_squarefree}.
Table~\ref{tab:longest_squarefree_instance} gives an explicit longest word for each such case.
These values and words were obtained with software tools\footnote{\url{https://github.com/koeppl/squarechecker}}.
We list alphabet sizes only up to $\sigma=5$.
For the finite entries in the tables, our experiments also gave the same values when $\sigma \ge 6$.

As a direction for future work, it would be interesting to study repetition thresholds~\cite{dejean72repetition,Badkobeh14finiterepetition} under these equivalence relations.
A repetition threshold is the smallest real number $r$ such that some infinite word avoids every factor of exponent greater than $r$.
This can be viewed as another direction in generalizing square-freeness: while $\ell^+$-square-freeness relaxes classical square-freeness by considering only long squares, repetition thresholds relax it by bounding the exponent of repetitions.
In the strict-equality setting, repetition thresholds have been studied extensively, including questions about the maximum exponent and the number of occurrences of factors attaining it.
Whether analogous results hold under other equivalence relations is an interesting question.

\begin{table}[t]
    \centering
    \setlength{\tabcolsep}{8pt}
    \caption{
    Examples of longest $\ell^+$-square-free words under various equivalence relations and alphabet sizes.
    Each finite entry gives one longest word.
    An entry marked $\infty$ indicates the existence of an infinite $\ell^+$-square-free word.
    A waved line indicates a conjectured result based on experimental results.
    A dash ($-$) indicates that infiniteness follows from another case (e.g., a smaller $\ell$ or $\sigma$).
    }
    \vspace{1em}

    \begin{minipage}[t]{\textwidth}
        \centering
        \textbf{Strict Equality}\\[0.5em]
        \begin{tabular}{c|cccc}
         $\sigma\backslash\ell$ & 1 & 2 & 3 & 4 \\
         \hline
         2 & ${\rm aba}$ & ${\rm abaabbaaabbbaabbab}$ & $\infty$ & $-$ \\
         3 & $\infty$ & $-$ & $-$ & $-$ \\
         4 & $-$ & $-$ & $-$ & $-$ \\
         5 & $-$ & $-$ & $-$ & $-$ \\
        \end{tabular}
    \end{minipage}

    \vspace{2em}

    \begin{minipage}[t]{\textwidth}
        \centering
        \textbf{Parameterized Equivalence}\\[0.5em]
        \begin{tabular}{c|cccc}
         $\sigma\backslash\ell$ & 1 & 2 & 3 & 4 \\
         \hline
         2 & ${\rm a}$ & ${\rm aaabaaa}$ & $\infty$ & $-$ \\
         3 & ${\rm a}$ & ${\rm aabaaabcc}$ & $-$ & $-$ \\
         4 & ${\rm a}$ & ${\rm aabaaabcc}$ & $-$ & $-$ \\
         5 & ${\rm a}$ & ${\rm aabaaabcc}$ & $-$ & $-$ \\
        \end{tabular}
    \end{minipage}

    \vspace{2em}

    \begin{minipage}[t]{\textwidth}
        \centering
        \textbf{Order-Preserving Equivalence}\\[0.5em]
        \begin{tabular}{c|cccc}
         $\sigma\backslash\ell$ & 1 & 2 & 3 & 4 \\
         \hline
         2 & ${\rm a}$ & ${\rm aaabaaa}$ & $\infty$ & $-$ \\
         3 & ${\rm a}$ & $\infty$ & $-$ & $-$ \\
         4 & ${\rm a}$ & $-$ & $-$ & $-$ \\
         5 & ${\rm a}$ & $-$ & $-$ & $-$ \\
        \end{tabular}
    \end{minipage}

    \vspace{2em}

    \begin{minipage}[t]{\textwidth}
        \centering
        \textbf{Cartesian-Tree Equivalence}\\[0.5em]
        \begin{tabular}{c|cccc}
         $\sigma\backslash\ell$ & 1 & 2 & 3 & 4 \\
         \hline
         2 & ${\rm a}$ & ${\rm baaba}$ & ${\rm bbaaabbabaabbbababaaabbababaa}$ & \uwave{$\infty$} \\
         3 & ${\rm a}$ & ${\rm aacbaccba}$ & \uwave{$\infty$} & $-$ \\
         4 & ${\rm a}$ & ${\rm aacbaccba}$ & $-$ & $-$ \\
         5 & ${\rm a}$ & ${\rm aacbaccba}$ & $-$ & $-$ \\
        \end{tabular}
    \end{minipage}
    \label{tab:longest_squarefree_instance}
\end{table}

\subsubsection*{\ackname}
We thank Pascal Ochem for pointing out relevant prior work.

\clearpage
\bibliographystyle{unsrt}
\bibliography{literature}

@Article{	  baker96parametrized,
  Author	= {Brenda S. Baker},
  DOI		= {10.1006/jcss.1996.0003},
  Journal	= {J. Comput. Syst. Sci.},
  Number	= {1},
  Pages		= {28--42},
  Title		= {Parameterized Pattern Matching: Algorithms and
		  Applications},
  Volume	= {52},
  Year		= {1996}
}

@Article{	  crochemore95squares,
  Author	= {Maxime Crochemore and Wojciech Rytter},
  DOI		= {10.1007/BF01190846},
  Journal	= {Algorithmica},
  Number	= {5},
  Pages		= {405--425},
  Title		= {Squares, Cubes, and Time-Space Efficient String
		  Searching},
  Volume	= {13},
  Year		= {1995}
}

@Article{	  currie18unary,
  Author	= {James D. Currie and Florin Manea and Dirk Nowotka and
		  Kamellia Reshadi},
  DOI		= {10.1016/j.tcs.2018.05.033},
  Journal	= {Theor. Comput. Sci.},
  Pages		= {72--82},
  Title		= {Unary patterns under permutations},
  Volume	= {743},
  Year		= {2018}
}

@Article{	  entringer74nonrepetitive,
  Author	= {Roger C. Entringer and Douglas E. Jackson and J. A.
		  Schatz},
  DOI		= {10.1016/0097-3165(74)90041-7},
  Journal	= {J. Comb. Theory {A}},
  Number	= {2},
  Pages		= {159--164},
  Title		= {On Nonrepetitive Sequences},
  Volume	= {16},
  Year		= {1974}
}

@Article{	  fraenkel98square,
  Author	= {Aviezri S. Fraenkel and Jamie Simpson},
  DOI		= {10.1006/jcta.1997.2843},
  Journal	= {J. Comb. Theory, Ser. {A}},
  Number	= {1},
  Pages		= {112--120},
  Title		= {How Many Squares Can a String Contain?},
  Volume	= {82},
  Year		= {1998}
}

@Article{	  kim14orderpreserving,
  Author	= {Jinil Kim and Peter Eades and Rudolf Fleischer and
		  Seok{-}Hee Hong and Costas S. Iliopoulos and Kunsoo Park
		  and Simon J. Puglisi and Takeshi Tokuyama},
  DOI		= {10.1016/j.tcs.2013.10.006},
  Journal	= {Theor. Comput. Sci.},
  Pages		= {68--79},
  Title		= {Order-preserving matching},
  Volume	= {525},
  Year		= {2014}
}

@Article{	  kociumaka16maximum,
  Author	= {Tomasz Kociumaka and Jakub Radoszewski and Wojciech Rytter
		  and Tomasz Walen},
  DOI		= {10.1016/j.tcs.2016.08.010},
  Journal	= {Theor. Comput. Sci.},
  Pages		= {84--95},
  Title		= {Maximum number of distinct and nonequivalent nonstandard
		  squares in a word},
  Volume	= {648},
  Year		= {2016}
}

@InProceedings{	  park19cartesian,
  Author	= {Sung Gwan Park and Amihood Amir and Gad M. Landau and
		  Kunsoo Park},
  BookTitle	= {Proc.\ CPM},
  DOI		= {10.4230/LIPIcs.CPM.2019.16},
  Pages		= {16:1--16:14},
  Series	= {LIPIcs},
  Title		= {{Cartesian} Tree Matching and Indexing},
  Volume	= {128},
  Year		= {2019}
}

@Article{	  rampersad05avoiding,
  Author	= {Narad Rampersad and Jeffrey O. Shallit and Ming{-}wei
		  Wang},
  DOI		= {10.1016/J.TCS.2005.01.005},
  Journal	= {Theor. Comput. Sci.},
  Number	= {1},
  Pages		= {19--34},
  Title		= {Avoiding large squares in infinite binary words},
  Volume	= {339},
  Year		= {2005}
}

@article{thue1906,
  author    = {Axel Thue},
  title     = {{\"U}ber unendliche Zeichenreihen},
  journal   = {Norske Vid. Selsk. Skr. Mat. Nat. Kl.},
  volume    = {7},
  pages     = {1--22},
  year      = {1906},
  note      = {Reprinted in \textit{Selected mathematical papers of Axel Thue}, T. Nagell (ed.), Universitetsforlaget, Oslo, 1977, pp.~139--158}
}

@inproceedings{keranen92abelian,
  author       = {Veikko Ker{\"{a}}nen},
  title        = {Abelian Squares are Avoidable on 4 Letters},
  booktitle    = {Automata, Languages and Programming, 19th International Colloquium,
                  ICALP92, Vienna, Austria, July 13-17, 1992, Proceedings},
  series       = {Lecture Notes in Computer Science},
  volume       = {623},
  pages        = {41--52},
  publisher    = {Springer},
  year         = {1992},
  doi          = {10.1007/3-540-55719-9\_62},
  bibsource    = {dblp computer science bibliography, https://dblp.org}
}

@article{cori90abelian,
  author       = {Robert Cori and
                  Maria Rosaria Formisano},
  title        = {Partially Abelian squarefree words},
  journal      = {{RAIRO} Theor. Informatics Appl.},
  volume       = {24},
  pages        = {509--520},
  year         = {1990},
  doi          = {10.1051/ITA/1990240605091},
  bibsource    = {dblp computer science bibliography, https://dblp.org}
}

@article{prodinger79infinite,
  author       = {Helmut Prodinger and Friedrich J. Urbanek},
  title        = {Infinite 0-1-sequences without long adjacent identical blocks},
  journal      = {Discret. Math.},
  volume       = {28},
  number       = {3},
  pages        = {277--289},
  year         = {1979},
  url          = {https://doi.org/10.1016/0012-365X(79)90135-3},
  doi          = {10.1016/0012-365X(79)90135-3},
  bibsource    = {dblp computer science bibliography, https://dblp.org}
}

@article{currie21characterization,
  author       = {James D. Currie and Jesse T. Johnson},
  title        = {Characterization of the lengths of binary circular words containing no squares other than 00, 11, and 0101},
  journal      = {Theor. Comput. Sci.},
  volume       = {850},
  pages        = {30--39},
  year         = {2021},
  url          = {https://doi.org/10.1016/j.tcs.2020.10.029},
  doi          = {10.1016/J.TCS.2020.10.029},
  bibsource    = {dblp computer science bibliography, https://dblp.org}
}

@article{prodinger83nonrepetitive,
  author       = {Helmut Prodinger},
  title        = {Non-repetitive sequences and {Gray} code},
  journal      = {Discret. Math.},
  volume       = {43},
  number       = {1},
  pages        = {113--116},
  year         = {1983},
  url          = {https://doi.org/10.1016/0012-365X(83)90027-4},
  doi          = {10.1016/0012-365X(83)90027-4},
  bibsource    = {dblp computer science bibliography, https://dblp.org}
}

@article{dejean72repetition,
  author       = {Fran{\c{c}}oise Dejean},
  title        = {Sur un Th{\'{e}}or{\`{e}}me de Thue},
  journal      = {J. Comb. Theory {A}},
  volume       = {13},
  number       = {1},
  pages        = {90--99},
  year         = {1972},
  url          = {https://doi.org/10.1016/0097-3165(72)90011-8},
  doi          = {10.1016/0097-3165(72)90011-8},
  bibsource    = {dblp computer science bibliography, https://dblp.org}
}

@article{Badkobeh14finiterepetition,
  author       = {Golnaz Badkobeh and
                  Maxime Crochemore and
                  Micha{\"{e}}l Rao},
  title        = {Finite repetition threshold for large alphabets},
  journal      = {{RAIRO} Theor. Informatics Appl.},
  volume       = {48},
  number       = {4},
  pages        = {419--430},
  year         = {2014},
  url          = {https://doi.org/10.1051/ita/2014017},
  doi          = {10.1051/ITA/2014017},
  bibsource    = {dblp computer science bibliography, https://dblp.org}
}

@inproceedings{hamai24parameterized,
  author       = {Rikuya Hamai and
                  Kazushi Taketsugu and
                  Yuto Nakashima and
                  Shunsuke Inenaga and
                  Hideo Bannai},
  editor       = {Zsuzsanna Lipt{\'{a}}k and
                  Edleno Silva de Moura and
                  Karina Figueroa and
                  Ricardo Baeza{-}Yates},
  title        = {On the Number of Non-equivalent Parameterized Squares in a String},
  booktitle    = {String Processing and Information Retrieval - 31st International Symposium,
                  {SPIRE} 2024, Puerto Vallarta, Mexico, September 23-25, 2024, Proceedings},
  series       = {Lecture Notes in Computer Science},
  volume       = {14899},
  pages        = {174--183},
  publisher    = {Springer},
  year         = {2024},
  url          = {https://doi.org/10.1007/978-3-031-72200-4\_13},
  doi          = {10.1007/978-3-031-72200-4\_13},
  bibsource    = {dblp computer science bibliography, https://dblp.org}
}

@article{manea15cubic,
  author       = {Florin Manea and
                  Mike M{\"{u}}ller and
                  Dirk Nowotka},
  title        = {Cubic patterns with permutations},
  journal      = {J. Comput. Syst. Sci.},
  volume       = {81},
  number       = {7},
  pages        = {1298--1310},
  year         = {2015},
  url          = {https://doi.org/10.1016/j.jcss.2015.04.001},
  doi          = {10.1016/J.JCSS.2015.04.001},
  bibsource    = {dblp computer science bibliography, https://dblp.org}
}

@inproceedings{NgORS19,
  author       = {Tim Ng and
                  Pascal Ochem and
                  Narad Rampersad and
                  Jeffrey O. Shallit},
  editor       = {Robert Mercas and
                  Daniel Reidenbach},
  title        = {New Results on Pseudosquare Avoidance},
  booktitle    = {Combinatorics on Words - 12th International Conference, {WORDS} 2019,
                  Loughborough, UK, September 9-13, 2019, Proceedings},
  series       = {Lecture Notes in Computer Science},
  volume       = {11682},
  pages        = {264--274},
  publisher    = {Springer},
  year         = {2019},
  url          = {https://doi.org/10.1007/978-3-030-28796-2\_21},
  doi          = {10.1007/978-3-030-28796-2\_21},
  bibsource    = {dblp computer science bibliography, https://dblp.org}
}

@article{Walnut,
  author       = {Hamoon Mousavi},
  title        = {Automatic Theorem Proving in Walnut},
  journal      = {CoRR},
  volume       = {abs/1603.06017},
  year         = {2016},
  url          = {http://arxiv.org/abs/1603.06017},
  eprinttype   = {arXiv},
  eprint       = {1603.06017},
  bibsource    = {dblp computer science bibliography, https://dblp.org}
}

\end{document}